\documentclass[12pt]{article}
\usepackage{authblk}
\usepackage{graphics}
\usepackage{graphicx}
\usepackage[mathscr]{eucal}
\usepackage{epic,eepicemu,eepic}
\usepackage{amsmath,amsxtra,amssymb,latexsym,amscd,amsthm,amsfonts}
\usepackage[left=3cm,right=2.5cm,top=2.5cm,bottom=2.5cm]{geometry}
\newtheorem{thm}{Theorem}
\newtheorem{cor}[thm]{Corollary}

\newtheorem{lem}[thm]{Lemma}

\theoremstyle{definition}
\newtheorem{rem}[thm]{Remark}

\begin{document}
\baselineskip=17pt
\title{\bf  Zero $f$-mean curvature surfaces of revolution in  the Lorentzian product $\Bbb G^2\times\Bbb R_1$}
 \small{}
\author[a]{ Doan The Hieu}
\author[b]{Tran Le Nam}
\affil[a]{College of Education, Hue University, Hue, Vietnam}
\affil[b]{Dong Thap University, Dong Thap, Vietnam}
 \maketitle
\begin{abstract} We classify (spacelike or timelike) surfaces  of revolution with zero $f$-mean curvature in $\Bbb G^2\times\Bbb R_1,$ the Lorentz-Minkowski 3-space  $\Bbb R^3_1$ endowed with the Gaussian-Euclidean density $e^{-f(x,y,z)}=\frac 1{2\pi}e^{-\frac{x^2+y^2}2}.$
  It is proved that  an $f$-maximal surface of revolution is either  a horizontal plane  or a spacelike $f$-Catenoid. For the timelike case, a timelike $f$-minimal surface is either a vertical plane containing $z$-axis, the cylinder $x^2+y^2=1,$  or a timelike $f$-Catenoid. Spacelike and timelike $f$-Catenoids are new examples of $f$-minimal surfaces in $\Bbb G^2\times \Bbb R_1.$
 \end{abstract}
 \vskip 1cm

%====================================
\section{Introduction}

In $\Bbb R^3,$  together with the plane, Catenoid is the only minimal surface of revolution.
If not counting the plane, it is the first minimal surface  discovered by Leonhard Euler in 1744.
The counterpart of minimal surfaces in the Lorentz-Minkowski space $\Bbb R^3_1,$ are  (spacelike or timelike) surfaces with zero mean curvature. 
Since the metric in $\Bbb R^3_1$ is not positive definite, there are three types of vectors (spacelike, lightlike or timelike). Therefore, more complicated than rotations in Euclidean space, in $\Bbb R^3_1,$ there are three types of Lorentzian rotations depending on the causal of the rotation axises. Maximal surfaces of revolution in $\Bbb R^3_1$ have been classified in \cite{ko}.  Spacelike and timelike surfaces of revolution with constant mean curvature in $\Bbb R^3_1$ have been studied in \cite{leva}, \cite{leva1} and  \cite{lema}.  Recently, maximal surfaces in Lorentzian product spaces have been also studied by some authors (see, for example, \cite{al}, \cite{alal1}, \cite{alal2}, \cite{li} and \cite{li1}).

It is natural to study (spacelike or timelike) surfaces of revolution with zero weighted mean curvature, also called $f$-mean curvature, in $\Bbb R^3_1$ endowed with a density, i.e.,  a positive function defined on $\Bbb R^3_1$ used to weight the area  (the length) of surfaces (curves). 

In this paper, such a density that we considered is the Gaussian-Euclidean density, i.e., the space is the Lorentzian product $\Bbb G^2\times \Bbb R_1,$  where $\Bbb G^2$ is the Gauss plane.
The space $\Bbb G^2\times\Bbb R_1$ is a special case of $n$-dimensional spacetime with a density that does not affect ``time''. It should be mentioned  that the space we are living can be seen as a  4-dimensional spacetime with density, the gravity,  that affect  ``space'' and does not affect ``time''.

It is showed that the axis of an $f$-maximal surface of revolution in $\Bbb G^2\times\Bbb R_1$ must be the $z$-axis. Then, solving the $f$-Maximal Surface Equation for surfaces of revolution we obtain new non-trivial examples,  called spacelike $f$-Catenoids. Beside horizontal planes, they are the only $f$-maximal surfaces of revolution.  This is the first result of the paper. 

For the timelike case, by a similar proof, it is proved that  the axis of a timelike $f$-minimal surface of revolution must be the $x$-axis or the  $z$-axis. If the rotation axis is the $x$-axis, the only timelike $f$-minimal surfaces of revolution are vertical planes containing the $z$-axis.
If the rotation axis  is the $z$-axis, there are a family of  timelike $f$-minimal surfaces of revolution, called timelike $f$-Catenoids, that convergences to another timelike $f$-minimal surface, the cylinder $x^2+y^2=1.$   The second main result of the paper is that a timelike $f$-minimal surface of revolution is either the cylinder $x^2+y^2=1,$ a vertical plane containing the $z$-axis or a timelike $f$-Catenoid. 

%====================
\section{Preliminaries} 
For simplicity, all concepts as well as results in this section are introduced in 3-dimensional space. For more details about Lorentz-Minkowski 
spaces, manifolds with density or the Gauss space we refer the reader to \cite{lo1}, \cite{mo1}, \cite{mo2}, \cite{nei}, \cite{we} and references therein. 

Let $\Bbb R^{3}_1$ be the Lorentz-Minkowski 3-space endowed with the Lorentzian scalar product
$$\langle \ , \ \rangle=dx^2+dy^2-dz^2.$$
A nonzero
vector $\textbf x\in \Bbb R^{3}_1$ is called spacelike, lightlike  or timelike    if  $\langle \textbf x, \textbf
x\rangle>0$, \ $\langle \textbf x, \textbf x\rangle=0,$ or  $\langle \textbf x, \textbf x\rangle<0,$ respectively.

The norm of the vector $\textbf x$ is then defined by $\|\textbf x\|=\sqrt{|\langle\textbf x, \textbf x\rangle|}.$
Two vectors $\textbf x_1=(x_1,y_1, z_1),\  \textbf x_2=(x_2, y_2, z_2)\in \Bbb R^{3}_1$  are said to be orthogonal if $\langle \textbf x_1, \textbf x_2\rangle=0,$ i.e., $x_1x_2+y_1y_2-z_1z_2=0.$ The 
 Lorentzian vector product of $\textbf x_1$ and $\textbf x_2,$  denoted by $\textbf x_1\wedge\textbf x_2,$  is defined by
$$\textbf x_1\wedge\textbf x_2=\begin{vmatrix} \textbf e_1&\textbf e_2&-\textbf e_3\\
x_1&y_1&z_1\\
x_2&y_2&z_2\end{vmatrix}, $$
where  $\{\textbf e_1,\textbf e_2,\textbf e_3\}$ is the canonical basis of $\Bbb R^{3}_1.$
For every $\textbf x\in \Bbb R^{3}_1,$
$$\langle \textbf x, \textbf x_1\wedge\textbf x_2\rangle=\det( \textbf x, \textbf x_1,\textbf x_2).$$
It follows that $\textbf x_1\wedge\textbf x_2$ is orthogonal to both
$\textbf x_1$ and $\textbf x_2.$

%================================================

A surface  in  $\Bbb R^{3}_1$ is called spacelike (timelike) if its induced metric
from $\Bbb R^{3}_1$ is Riemannian (Lorentzian) or equivalently,  every normal vector of the surface is timelike (spacelike).   

For example, let $\alpha$ be a plane whose general equation is $Ax+By+Cz+D=0, \ A^2+B^2-C^2\ne 0.$ It is easy to see that, the vector 
${\bf n}=(A,B,-C)$ is a  normal vector of $\alpha.$ The plane $\alpha$ is spacelike or timelike if and only if $\bf n$ is timelike or spacelike, respectively.

The following formula for computing the mean curvature of a (spacelike or timelike) surface in $\Bbb R^3_1$  is well-known (see \cite{lo2}, for instance)
\begin{align}\label{mean}
H=\epsilon\frac{Eg-2 Ff+Ge}{2(EG-F^2)},
\end{align}
where $\epsilon=-1,$ if the surface is spacelike; $\epsilon=1,$ if the surface is timelike; $E, F, G$ are the coefficients of the first fundamental form and $e,f,g$  are the coefficients of the second fundamental form.

A spacelike (timelike) surface is called maximal (timelike minimal) if its mean curvature $H$ is zero everywhere.  

There are three kinds of rotations in $\Bbb R^3_1$: rotations about a spacelike axis, rotations about a timelike axis and rotations about a lightlike axis  (see \cite{leva}, for instance).  Below are the matrices of some typical kinds of rotations that will be used in the proof of Lemma \ref{sf} and Lemma \ref{tf}.
\begin{enumerate}
\item The matrix corresponding to a rotation about the $y$-axis is 
$$	\begin{pmatrix}
	\cosh \theta & 0 & \sinh\theta\\
	0&1&0\\
\sinh\theta & 0& \cosh \theta
\end{pmatrix}. $$
\item The matrix corresponding to a rotation about the lightlike axis  $x=z,\ y=0$  is 
\begin{align*}
	\begin{pmatrix}
	1-\dfrac{v^2}{2}&-v&\dfrac{v^2}{2}\\
	v&1&-v\\
	\dfrac{-v^2}{2}&-v&1+\dfrac{v^2}{2}
	\end{pmatrix}.
\end{align*}
\item The matrix corresponding to a rotation about the $z$-axis is 
$$	\begin{pmatrix}
	\cos\theta &-\sin\theta &0\\
\sin\theta &  \cos \theta &0\\
0&0&1\\
\end{pmatrix}. $$
\end{enumerate}

%=============================
A surface of revolution is a surface in $\Bbb R^3_1$ obtained by rotating a  curve $\gamma,$ the generatrix, around an axis of rotation $l,$  assuming that $\gamma$ and $l$ are in a plane. 

A density on $\Bbb R^{3}_1$ is a positive function, denoted by $e^{-f},$  used to weight area (length) of surfaces (curves).
The weighted mean curvature or the $f$-mean curvature of a (spacelike or timelike) surface,  denoted by $H_f,$ is defined by 
$$H_f=H+\frac 12\langle  \nabla f, N\rangle.$$
 A spacelike (timelike) surface $\Sigma$ is called $f$-maximal (timelike $f$-minimal) if  $H_f=0$ everywhere, i.e., $H=-\frac 12\langle \nabla f, N\rangle.$ 

Gauss space $\Bbb G^2,$ is  just $\Bbb R^2$ with the Gaussian probability density
$$e^{-f(x,y)}=\frac 1{2\pi}e^{-\frac{x^2+y^2}2},$$
where $ (x,y)\in\Bbb G^2.$
%Gauss space has many applications to probability and statistics (see \cite{mo1}, \cite{mo2}).
  
Therefore, the Lorentzian product $\Bbb G^2\times\Bbb R_1$ can be seen as  $\Bbb R^{3}_1=\Bbb R^2\times\Bbb R_1$ endowed with the Gaussian-Euclidean density
$$e^{-f(x,y,z)}=\frac 1{2\pi}e^{-\frac{x^2+y^2}2},$$
where $ (x,y,z)\in\Bbb G^2\times \Bbb R_1.$ It should be noted that the last coordinate is not dependent on the density.

%=========================================================
%\subsection{Geometric meaning of the quantity $\langle\nabla f,\mathbf{N}\rangle.$}
Let $\Sigma$ be an oriented  (spacelike or timelike) surface in $\mathbb{G}^2\times\mathbb{R}_1,$ \ $N$ be a unit normal vector field  on $\Sigma$ and 
 $\rho$  be the projection onto  the $z$-axis. Then  at any point $p\in \Sigma,$  we have the following.
\begin{lem}\label{lem1}(Geometric meaning of the quantity $\langle\nabla f,\mathbf{N}\rangle$)
\begin{align*}
%\label{eq:geometric}
\left| \langle\nabla f,\mathbf{N}\rangle (p)\right|=\operatorname{d_E}\big(\rho(p),T_p\Sigma\big).|\mathbf{N}|_{\operatorname{E}},
\end{align*}
where $d_E$ and $| \ \  |_E$ denote the Euclidean distance the Euclidean length, respectively.
\end{lem}
\begin{proof}
Suppose that $p=(x_0, y_0, z_0)$ and $\mathbf{N}(p)=(a,b,c),\ a^2+b^2-c^2=\pm 1.$ An equation of $T_p\Sigma$ is of the form $ax+by-cz+d=0.$
We have $\nabla f(p)=(x_0,y_0,0)$ and $\rho(p)=(0,0,z_0).$ Therefore 
$$\left| \langle\nabla f,\mathbf{N}\rangle (p)\right|=|ax_0+by_0|=|cz_0-d|=\operatorname{d_E}\big(\rho(p),T_p\Sigma\big).|\mathbf{N}|_{\operatorname{E}}.$$
\end{proof}
By Lemma \ref{lem1},  it is not hard to prove the followings.
\begin{cor} \label{co1} In  $\mathbb{G}^2\times\mathbb{R}_1,$
\begin{enumerate}
\item horizontal planes  are $f$-maximal surfaces; 
\item vertical planes have constant $f$-mean curvature, such a plane containing the $z$-axis is timelike $f$-minimal; 
\item circular cylinders about the $z$-axis  have constant $f$-mean curvature, such a cylinder is timelike $f$-minimal if and only if the radius is 1.
\end{enumerate}
\end{cor}
%\begin{rem}
%Lemma \ref{lem1} holds true also for general case $\Bbb G^n\times \Bbb R_1.$ 
%\end{rem}
%==============================================
\section{Spacelike $f$-maximal surfaces of revolution in $\Bbb G^2~\times~\Bbb R_1$}
\subsection{Spacelike $f$-Catenoids in $\Bbb G^2\times\Bbb R_1$}

 In the $xz$-plane, consider the curve $\gamma_S$ that is the graph of the function (see Figure 1).
$$h(u)= \int_{0}^{u}\sqrt{\dfrac{1}{1+\tau^2e^{\tau^2+C}}}d\tau,\ \ u\in\Bbb R.$$
%\left(u, 0, \int_{0}^{u}\sqrt{\dfrac{1}{1+\tau^2e^{\tau^2+c}}}d\tau\right), \ u\in \Bbb R

 Rotating  $\gamma_S$ the about the $z$-axis, we obtain a surface of revolution (see Figure 2), denoted by  $\Sigma_S,$  that can be parametrized as follows.
$$	X(u,v)=\left(u\cos v, u\sin v, \int_{0}^{u}\sqrt{\dfrac{1}{1+\tau^2e^{\tau^2+C}}}d\tau\right),$$
where $C$ is a constant. It is easy to verify that the curve is spacelike and therefore the surface $\Sigma_S$ is spacelike. The surface $\Sigma_S$ has a singular point, that is the origin. 
By a direct computation, it follows that the $f$-mean curvature  of  $\Sigma_S$ is zero, i.e., $\Sigma_S$ is $f$-maximal. We call $\Sigma_S$  a spacelike $f$-Catenoid. 
\vskip 1cm
\begin{center}
	\includegraphics{DuongSinhSpacelike-1}\\
	\textit{Figure 1.\ The generatrix $\gamma$}
\end{center}

\begin{center}
	\includegraphics{DuongSinhSpacelike-2}\\
	\textit{Figure 2.\ Spacelike $f$-Catenoid.}
\end{center}
		
%	\end{tabular}
%\end{center}
\subsection{Classification of $f$-maximal surfaces of revolution in $\Bbb G^2\times\Bbb R_1$}
In the Lorentz-Minkowski space $\Bbb R^3_1,$ because the mean curvature of a (spacelike or timelike) surfaces is invariant under Lorentzian transformation, when studying surfaces of revolution of constant mean curvature, if the rotation axis is timelike, spacelike or lightlike we can suppose it is the $z$-axis, the $x$-axis, or the lightlike axis $x=z, y=0,$ respectively. In the space $\Bbb G^2\times \Bbb R_1,$ with the appearance of the density, we can not do this because the $f$-mean curvature is not invariant under some Lorentzian transformations. Since the density is dependent on the distance from points to the $z$-axis and not dependent on the last coordinate, the $f$-mean curvature of a surface does not change under rotations about as well as  translations along the $z$-axis (see Lemma (\ref{lem1}). This observation is useful for the rest of the paper to simplify some calculations.

%==============================================
\begin{lem}\label{sf} A spacelike surface of revolution $\Sigma$ in $\Bbb G^2\times\Bbb R_1$ can be parametrized as follows.
\begin{enumerate}
\item If the rotation axis is spacelike 
\begin{align}\label{ss} X(u,v)=(u\cosh\theta+g(u)\sinh \theta\cosh v, u\sinh v+a, u\sinh\theta+g(u)\cosh\theta\cosh v).
\end{align}
\item If the rotation axis is lightlike
\begin{align}\label{ls}
X(u,v)=\left(u-\left[u-g(u)\right]\dfrac{v^2}{2}+a,v[u-g(u)],g(u)-\left[u-g(u)\right]\dfrac{v^2}{2}\right).
\end{align}
\item If the rotation axis is timelike
\begin{align}\label{ts}
X(u,v)=(u\cos v\cosh\theta+g(u)\sinh\theta, u\sin v +a,u\cos v\sinh\theta+g(u)\cosh\theta).
\end{align}
\end{enumerate}
\end{lem}

\begin{proof}
\begin{enumerate}
\item {\bf The case $l$ is spacelike}.
Under a suitable rotation about the $z$-axis,  we can assume that the plane containing the generatrix $\gamma$ and the rotation axis $l$ are parallel to or coincident with the $xz$-plane.  If $l$ and the $z$-axis are not intersect, we assume that the common perpendicular line of $l$ and the $z$-axis is the $y$-axis. If  $l$ and the $z$-axis are intersect, we assume that the intersection point is the origin $O.$ Let $\{H\}=(0,a,0)$ be the intersection point of $l$ and $xy$-plane and let  $\theta$ be the angle between $l$ and $Ox.$ 

There exist a Lorentz transformation that maps $\Sigma$ to a surface of revolution $\Sigma_1$ obtained by rotating a spacelike $\gamma_1,$ that lies in the $xz$-plane, about the $x$-axis. 
This transformation is a composition of a translation along $y$-axis by a vector $v=(0, a, 0)$ and a rotation about $y$-axis of angle $\theta.$ 
Because the curve $\gamma_1$ is spacelike, it can be parametrized as $\gamma_1(u)=(u,0,g(u)),\ u\in I\subset\Bbb R,\ g\ne 0,\ 1-g'^2>0.$ Then, a parametrization  of $\Sigma_1$ is $X(u,v)= (u, g(u)\sinh, g(u)\cosh v)$ and therefore a parametrization of $\Sigma$ is 
\begin{align*}
	X(u,v)&=\begin{pmatrix}
	\cosh \theta & 0 & \sinh\theta\\
	0&1&0\\
\sinh\theta & 0& \cosh \theta
\end{pmatrix}
\begin{pmatrix}
u\\ g(u)\sinh v\\
  g(u)\cosh v
\end{pmatrix} + \begin{pmatrix}
 0\\
a\\
 0
\end{pmatrix}\\
&=(u\cosh\theta+g(u)\sinh \theta\cosh v, g(u)\sinh v+a, g(u)\sinh\theta+g(u)\cosh\theta\cosh v)
\end{align*}
%============================================================
\item {\bf The case $l$ is lightlike}. 

By a suitable rotation about the $z$-axis, we can assume that the plane containing the generatrix $\gamma$ and the rotation axis $l$ is the $xz$-plane and $l$ is parallel to $e_1+e_3.$   Let $\{H\}=(a,0,0)$ be the intersection point of $l$ and the $xy$-plane and suppose that $\gamma(u)=(u, 0, g(u)).$ Then,  a parametrization of $\Sigma$ is

\begin{align*}
X(u,v)&=\begin{pmatrix}
	1-\dfrac{v^2}{2}&-v&\dfrac{v^2}{2}\\
	v&1&-v\\
	\dfrac{-v^2}{2}&-v&1+\dfrac{v^2}{2}
	\end{pmatrix}\begin{pmatrix}
u\\ 0\\ g(u)
\end{pmatrix} + \begin{pmatrix}
a\\
0\\
 0
\end{pmatrix}\\
&=\left(u-\left[u-g(u)\right]\dfrac{v^2}{2}+a,v[u-g(u)],g(u)-\left[u-g(u)\right]\dfrac{v^2}{2}\right).
\end{align*}
%==================================
\item {\bf The case $l$ is timelike}. 

By the same arguments as in the case $l$ is spacelike, but in this case, $\theta$  is the angle between $l$ and the $z$-axis and $\Sigma_1$ is the surface of revolution obtained by rotating $\gamma_1$ about the $z$-axis. 

A parametrization of $\Sigma_1$ is  $(u\cos v, u\sin v, g(u)).$ Therefore, a parametrization of $\Sigma$ is
\begin{align*}X(u,v)&=	\begin{pmatrix}
	\cosh \theta & 0 & \sinh\theta\\
	0&1&0\\
\sinh\theta & 0& \cosh \theta
\end{pmatrix}\begin{pmatrix}
u\cos v\\  u\sin v\\ g(u)
\end{pmatrix} + \begin{pmatrix}
 0\\
a\\
 0
\end{pmatrix}\\
&=(u\cos v\cosh\theta+g(u)\sinh\theta, u\sin v +a, u\cos v\sinh\theta+g(u)\cosh\theta).
\end{align*}
\end{enumerate}
\end{proof}
%===========================================
%=============================================
%If a spacelike surface of revolution $\Sigma$  generated by the curve $\gamma$ under the boost of axis $l$ is $f$-maximal, then $l$ must be the $z$-axis.
\begin{thm}\label{thm1} 
 An $f$-maximal surface of revolution $\Sigma$  in $\Bbb G^2\times\Bbb R_1$ is either a horizontal plane or a spacelike $f$-Catenoid.
 \end{thm}
\begin{proof}
 Since Lorentz transformations preserve the mean  curvature of surfaces,  along a coordinate curve $u=u_0,$  the mean curvature $H$ of $\Sigma$ is constant. 
Therefore if the $f$-mean curvature $H_f$ of $\Sigma$ is zero,  along any coordinate curve $u=u_0,$
$\langle \nabla f, N\rangle$ must be constant. This fact will be used to eliminate the case that the rotation axis $l$ is spacelike or lightlike. 
\end{proof}
%It should be noted that $H_f$ is not invariant under a Lorentz transformation in general.
The followings are obtained by straightforward computations.
\begin{itemize}
\item If $l$  is spacelike and (\ref{ss}) is a parametrization of $\Sigma,$ then
\begin{align*}%\label{eq2}
\langle \nabla f, N\rangle&=
\frac{1}{\sqrt{1-[g'(u)]^2}}[u\sinh\theta\cosh\theta\cosh v + g(u)(\sinh^2\theta\cosh^2v+\sinh^2v\\
&\;\;+g'(u)\sinh\theta\cosh v\cosh\theta)+ug'(u)\cosh^2\theta+a\sinh v{\sqrt{1-[g'(u)]^2}}].
\end{align*}
The condition ``$\langle \nabla f, N\rangle$ is not constant'' is equvalent to that ``for any $u$
$$Q:=\displaystyle\frac{\partial}{\partial v}\sqrt{1-[g'(u)]^2}\langle \nabla f, N\rangle,$$
must be zero for every $v.$'' 
By a straightforward computation, we obtain
\begin{align*}
Q
&=\dfrac{u\sinh2\theta}{2}\sinh v+ g(u)\left(\sinh^2\theta\sinh 2v+\sinh 2v+g'(u)\dfrac{\sinh2 \theta}{2}\sinh v \right) + a\cosh v\\
	&=g(u)\sinh 2v\cosh^2\theta+[u+g(u)g'(u)]\dfrac{\sinh2 \theta}{2}\sinh v+a\cosh v.
\end{align*}
It is not hard to see that  if  for any $u\in I,$\ $Q$  vanishes for every $v$ then $g(u)=0.$ This is impossible because $g\ne 0.$
\item If $l$  is lightlike  and (\ref{ls}) is a parametrization of $\Sigma$, then
\begin{align*}
	\langle \nabla f, N\rangle&=\frac 1{g'^2(u)-1}\left[u-\left[u-g(u)\right]\dfrac{v^2}{2}\right]\left[g'(u)+\dfrac{v^2}{2}\left[1-g'(u)\right]\right]+v^2[u-g(u)][g'(u)-1]\\
	&=\frac 1{{4}(g'^2(u)-1)}\left[(g(u)-u)(1-g'(u))v^4+2(u+g(u)g'(u)-2g(u))v^2\right]+ug'(u).
\end{align*}
	We can verify that $\langle \nabla f, N\rangle$ is not constant.
%=====================================
\item The case $l$ is timelike  and (\ref{ts}) is a parametrization of $\Sigma,$  
A direct computation shows that
\begin{align}
\label{eq2}
\langle \nabla f, N\rangle&=\dfrac{1}{\sqrt{1-[g'(u)]^2}}\left[ ug'(u)(1+\cos^2v\sinh^2\theta)+[u+g(u)g'(u)]\dfrac{\sinh2\theta}{2}\cos v+g(u)\sinh^2\theta\right]\nonumber\\
&\quad+\dfrac{ag'(u)\sin v}{\sqrt{1-[g'(u)]^2}}.
\end{align}
\end{itemize}
We can see that, for any $u,$
$\langle \nabla f, N\rangle$ is constant if and only if $\theta=a=0,$ i.e.,  $l$ must be the $z$-axis.
%============================================================
The parametrization of  $\Sigma$ is now become
\begin{align} 	X(u,v)=(u\cos v, u\sin v, g(u)).
\end{align} 

A direct computation shows that
\begin{align*}
	& H=\dfrac{-1}{2}\dfrac{(1-g'^2)g'+ug''}{u(1-g'^2)^{3/2}},\\
	&\langle \nabla f, N\rangle=\dfrac{-g'u}{\sqrt{1-g'^2}}.
\end{align*}
Therefore, $\Sigma$ is $f$-maximal if and only if $g$ satisfies the following equation
\begin{align}
\label{revolution1}
	(1-g'^2)g'+ug''+u^2g'(1-g'^2)=0.
\end{align}

We solve equation \eqref{revolution1}.
\begin{itemize}
\item It is clear that $g(u)=a,$  where $a$ is constant, is a solution of \eqref{revolution1}, i.e., $\Sigma$ is a vertical plane.
\item Now locally we can suppose that  $g'(u)\ne0$ for every $u\in J,$ where $J\subset I.$  Multiply both sides of \eqref{revolution1} by $g'$ and set $h=g'^2$, we get
\begin{align*}
- \dfrac{dh}{du}=2\dfrac{u^2+1}{u}h(1-h).
\end{align*}
Solving this equation, we obtain
\begin{align*}
 \ln\dfrac{1-h}{h}=u^2+\ln u^2 + C,\; C\in\mathbb{R},
\end{align*}
or
\begin{align*}
h=\dfrac{1}{1+u^2e^{u^2+C}}.
\end{align*}
Therefore,
\begin{align*}
 g'(u)=\pm \sqrt{\dfrac{1}{1+u^2e^{u^2+C}}},
\end{align*}
and
\begin{align*}
 g(u)=\pm\int_{u_0}^u\sqrt{\dfrac{1}{1+\tau^2e^{\tau^2+C}}}d\tau,\ \ u_0\in J.
\end{align*}
The function $g$ is defined over $\Bbb R,$  therefore we can assume that $I=\Bbb R, \  u_0=0,$ i.e.,
$\gamma\equiv\gamma_s$ or $\gamma\equiv\overline{\gamma_s},$ where $\overline{\gamma_s}$ is the graph of the function
$$\overline{h}(u)=-\int_{0}^u\sqrt{\dfrac{1}{1+\tau^2e^{\tau^2+C}}}d\tau,\ \ u\in \Bbb R.$$
It is clear that $\gamma_s$ and   $\overline{\gamma_s}$ generate the same surface of revolution $\Sigma_S.$
\end{itemize}

%=================================================================
\section{Timelike $f$-minimal surfaces of revolution in $\Bbb G^2~\times~\Bbb R_1$}

\subsection {Timelike $f$-Catenoids in $\Bbb G^2\times \Bbb R_1$}\label{4.1}

In the $xz$-plane consider the curve $\gamma_T$ that is the graph of the function 
$$h(u)=\int_{u_0}^u {\sqrt{\dfrac{e^{\tau^2}}{e^{\tau^2}-C\tau^2}}d\tau },$$
where $C$ is a positive constant and $u_0$ belongs to the domain $D$ of the function.
The domain $D$ is determined by the following lemma.
\begin{lem}\label{domain}
Consider the function $h:\Bbb R\longrightarrow\Bbb R$ defined by $h(u)=e^{u^2}-Cu^2,$ where $ C>0. $ Then
\begin{enumerate}
\item If $0<C<e,$ then $h(u)>0,\ \forall u\in \Bbb R.$ 
\item If $C=e,$ then $h(u)>0,\ \forall u\ne-1,1.$ 
\item If $C>e,$ then there exist $0<u_1<1< u_2,$ such that $h(u)>0,\ \forall u\in (-\infty, -u_2)\cup(-u_1,u_1)\cup(u_2,+\infty).$ 
\end{enumerate}
\end{lem}
\begin{proof}
Because the function $h$ is even, we just consider the case $u\ge 0.$ Taking the derivative of the function, we obtain
$$ h'(u)=2u\left(e^{u^2}-C\right).$$ 

\begin{enumerate}
 \item If  $ C\le 1,$ then $ h'(u)>0, \forall u>0.$ The function $h$ is  monotonically increasing  and therefore $ h(u)>0, \forall u\ge 0.$ Note that  $h(0)=1.$ 
\item If $C>1,$  the function has the only minimum point at $u=\sqrt{\ln C}$ and $h(\sqrt{\ln C})=C-C\ln C.$ We consider the following subcases.
\begin{itemize}
\item The case $1<C<e.$ Because $h(\sqrt{\ln C})=C-\ln C>0,$ \ $ h(u)>0, \forall u\ge 0.$
\item  The case $C=e.$ 	We can see that  $ h(u)>0, \forall u\ne 1$ and $h(1)=0.$
\item The case $C>e.$ Because $h(\sqrt{\ln C})=C-C\ln C<0,$ there exist two values  $0<u_1<1 <u_2$ such that \ $h(u_1)=h(u_2)=0,$ \ $h(u)>0, \ \forall u \notin [u_1, u_2]$ and $h(u)\le0, \ \forall u \in [u_1, u_2].$
\end{itemize}
\end{enumerate}

\begin{center}
	\includegraphics{BienThienLorent-1}
\end{center}
\end{proof}
By Lemma \ref{domain}, the domain $D$   and $u_0$ are chosen as follows.
\begin{enumerate}
\item If $0<C<e,$ then $D=\Bbb R$ and $u_0=0.$ 
\item If $C=e,$ then $D=(-\infty, -1), \ u_0=-1$ or $D=(-1,1),\ u_0=0$ or or $D=(1, +\infty), u_0=1.$
\item If $C>e,$ then $D=(-\infty, -u_2), \ u_0=-u_2$ or $D=(-u_1,u_1), \ u_0=0;$ or $D=(u_2, +\infty), u_0=u_2.$
\end {enumerate}
 Rotate the curve about the $z$-axis, we obtain a surface of revolution, denoted by  $\Sigma_T,$ that can be parametrized as follows.
$$	X(u,v)=\left(u\cos v, u\sin v, \int_{u_0}^u {\sqrt{\dfrac{e^{\tau^2}}{e^{\tau^2}-C\tau^2}}d\tau }\right).$$
 By a direct computation, it follows that the curve is timelike, $\Sigma_T$ is timelike. Moreover 
 $\Sigma_T$ is timelike $f$-minimal. We call $\Sigma_T$ a timelike $f$-Catenoid.

%============================

\begin{center}
	\includegraphics[width=7cm]{DuongsinhLorentz-1}\\
\vskip.3cm
	\textit{Figure 3. Generatrices corresponding to $ C=2, 1.5, 1, 0.5$ and the line $x=z,$ respectively}
\end{center}
%============================================

\begin{center}
		\includegraphics[width=4cm]{C2P5-1}\hspace*{2cm}
		 \includegraphics[width=4cm]{C2P5-2}\\
\vskip .1cm
	\textit{Figure 4. The generatrix and the corresponding timelike $f$-minimal surface ($C<e$)}
		
	\end{center}

%============================================

\begin{center}
	\includegraphics[width=4cm]{CBangE_e-1}\hspace*{1cm}	\includegraphics[width=3cm]{CBangE_e-2}\\
	\vskip .5cm
	\textit{Figure 5. The generatrix and the corresponding timelike $f$-minimal surface ($C=e,\ D=(-1,1)$)}
		\end{center}

\vskip 1cm
\begin{center}
	\includegraphics[width=6cm]{CBangE_e-3}\hskip1cm  %\includegraphics{CBangE_e-2}\\
	\includegraphics[width=6cm]{CBangE_e-4}%\hspace*{1cm}	\\%\includegraphics{CBangE_e-2}\\
	\vskip .5cm
	\textit{Figure 6. The generatrix and the corresponding timelike $f$-minimal surface ($C=e,\ D=(-\infty,-1)\cup(1, \infty)$)}
\end{center}

%==================================
%C>e
\begin{center}
	\includegraphics[width=6cm]{C3P1-1}\hspace*{0.5cm}
	\includegraphics[width=4cm]{C3P1-2}\\	
		\vskip .5cm
	\textit{Figure 7. The generatrix and the corresponding timelike $f$-minimal surface ($C=3.1,\ D=(-0.75,0.75)$)}
\end{center}
	\vskip .3cm
\begin{center}
	\includegraphics[width=5cm]{C3P1-8}\hspace*{0.5cm}
	\includegraphics[width=5cm]{C3P1-5}\\
		\vskip .5cm
	\textit{Figure 8. The generatrix and the corresponding timelike $f$-minimal surface ($C=3.1,\ D=(1.267,3)$)}
\end{center}

\begin{rem}
\begin{enumerate}
\item If $C=e,$ the integral $\displaystyle\int_{0}^{1}\sqrt{\dfrac{e^{\tau^2}}{e^{\tau^2}-C\tau^2}}d\tau $ is divergence. 
\begin{center}
	\includegraphics[width=10cm]{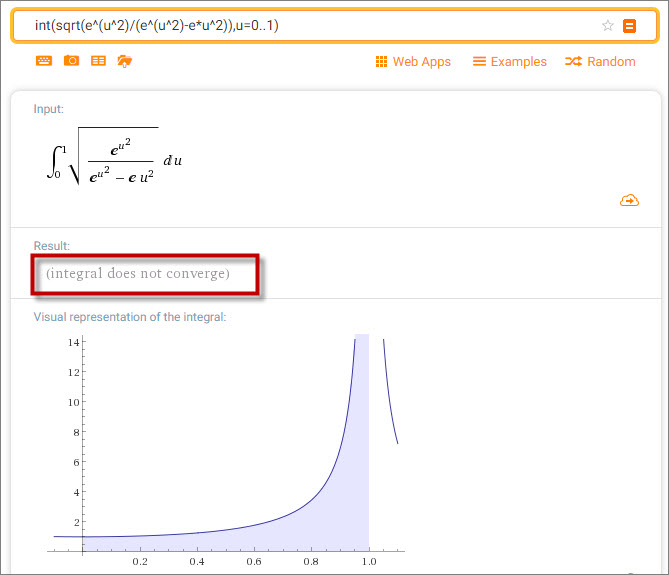}
	\vskip .3cm
	\textit{The divergence of  the integral is  showed by WolframAlpha}
\end{center}
\item If $ C>e= 2.718281828...,$ let $ u_1<1<u_2$ are solutions of the equation $ e^{u^2}-Cu^2=0,$ $\displaystyle I_1=\int_{0}^{u_1}\sqrt{\dfrac{e^{\tau^2}}{e^{\tau^2}-Cu^2}}d\tau $
and $\displaystyle I_2=\int_{u_2}^{4}\sqrt{\dfrac{e^{\tau^2}}{e^{\tau^2}-Cu^2}}d\tau.$ The following table computed by Maple gives us some specific values. We can see that the intergal $I_1$ is convergence. When $C$ goes to $\infty$ both $u_1$ and $I_1$ goes to 0. 
 \vskip 1cm
 \begin{tabular}{c|cccccccc}
 	$C$ &2.72&2.725 &2.73 &2.74 & 2.75\\
 	\hline
 	$u_1$&0.9822782644&0.9650767467 & 0.9539872965 &0.9375982937 &0.9248309636\\ 
 	\hline
 	$u_2$&1.017827051& 1.035334695 &1.046729750  & 1.063728407&1.077103331\\
 	\hline
 	$I_1$& 3.39646&2.913861676 &2.716714471 & 2.497786122 &2.363204279 \\
 	\hline
 	$I_2$&5.63324 &5.150340780 &4.952913186 & 4.733425412 &4.598285949
 \end{tabular}
 
 %\textcolor[rgb]{1,1,1}{.}\\
 
\begin{tabular}{c|cccccccc}
	$C$ &2.8  & 2.9 & 3.0 & 3.1 & 3.2\\
	\hline
	$u_1$&0.8808758710 & 0.8258522408 & 0.7868044780 &0.7556136794 & 0.7293528965\\ 
	\hline
	$u_2$&1.124065962 & 1.184957908 & 1.229688803 &1.266389104 & 1.297996253\\
	\hline
	$I_1$& 2.026032087  & 1.740252676 & 1.583315220 & 1.474783161&1.391943189\\
	\hline
	$I_2$&4.258352255 & 3.967178286 & 3.805008572 & 3.691396684 &3.603620332
\end{tabular}

%\textcolor[rgb]{1,1,1}{.}\\

\begin{tabular}{c|cccccccc}
	$C$ &4 & 5 & 6 &7 &8 \\
	\hline
	$u_1$& 0.5978318795&0.5090885010 & 0.4521962510 & 0.4113302857 & 0.3800280951\\
	\hline
	$u_2$& 1.467410087 & 1.594566197 & 1.683195738 & 1.751120026&1.806013755\\
	\hline
	$I_1$&  1.051169565& 0.8629326647 & 0.7526390978 & 0.6770807650&0.6208959900\\
	\hline
	$I_2$& 3.227665154& 3.003422822 &2.863129355 &2.761819500 &2.683055126
\end{tabular}

%\textcolor[rgb]{1,1,1}{.}\\

\begin{tabular}{c|cccccccc}
	$C$ &10 &20 &  30 &40 &50\\
	\hline
	$u_1$& 0.3344137545&  0.2295778377 &0.1857512382 & 0.1601547153 & 0.1428721249\\
	\hline
	$u_2$& 1.891336053 & 2.121262664& 2.239037675 & 2.317248453 &2.375356389\\
	\hline
	$I_1$&  0.5412131026 & 0.3655508690 & 0.2943554575 &0.2532125983&0.2255845814\\
	\hline
	$I_2$&  2.565108172 & 2.266933975 & 2.122055409 &2.027988230 &1.959035117
\end{tabular}

%\textcolor[rgb]{1,1,1}{.}\\

\begin{tabular}{c|cccccccc}
	$C$ &100 &200 &  300 &400 &500\\
	\hline
	$u_1$&  0.1005063540 & 0.07088856878 & 0.05783165509&0.05006269611&0.04476619308\\
	\hline
	$u_2$&  2.544164917 & 2.698888129 &2.784186390&2.842705995 &2.887054100\\
	\hline
	$I_1$& 0.1582764960 & 0.1114918791 & 0.09091788154&0.07868765650 & 0.07035385008\\
	\hline
	$I_2$&   1.762485425 &1.586306978 & 1.490482830 & 1.425201168 &1.375955282
\end{tabular}
\end{enumerate}
\end{rem}

 %==============================

\subsection{Classification of timelike $f$-minimal surfaces of revolution in $\Bbb G^2\times \Bbb R_1$}
Let $\Sigma$ be a timelike surface of revolution with the generatrix $\gamma$ and the rotation axis $l.$  
Since $\gamma$ is timelike, it can be expressed as below 
$$\gamma(u)=(g(u), 0,u).$$
 As in the case $\Sigma$ is spacelike, a local parametrization of $\Sigma$ as well as $\langle \nabla f, N\rangle$ can be computed as follows. 
\begin{lem}\label{tf} 
\begin{enumerate}
\item If the rotation axis is spacelike 
\begin{align*}%\label{ts} 
	X(u,v)=(g(u)\cosh\theta+u\sinh\theta\cosh v,u\sinh v+a,g(u)\sinh\theta+u\cosh\theta\cosh v).
\end{align*}
\begin{align}\label {nts}
\langle \nabla f, N\rangle&= \frac1{\sqrt{[g'(u)]^2-1}}[(g(u)g'(u)+u)\sinh\theta\cosh\theta\cosh v + g(u)\cosh^2\theta+ug'(u)\sinh^2\theta\cosh^2 \nonumber\\
&\quad +(g(u)\sinh v+a)(g'(u)\sinh v)].
\end{align}
%%%%%%%%%%%%%%%%%%%%%%
\item If the rotation axis is lightlike
\begin{align*}%\label{tl}
X(u,v)=\left(g(u)+\left[u-g(u)\right]\dfrac{v^2}{2}+a,-v[u-g(u)],u+\left[u-g(u)\right]\dfrac{v^2}{2}\right)\end{align*}
\begin{align}
\label{ntl}
\langle \nabla f, N\rangle&=\dfrac{1}{\sqrt{1-[g'(u)]^2}}\left[ ug'(u)(1+\sin^2v\sinh^2\theta)+[u+g(u)g'(u)]\dfrac{\sinh2\theta}{2}\sin v+g(u)\sinh^2\theta\right]\nonumber\\
&\quad+\dfrac{ag'(u)\cos v}{\sqrt{1-[g'(u)]^2}}.
\end{align}

%%%%%%%%%%%%%%%%%%%%%%%%%%%
\item If the rotation axis is timelike
\begin{align*}%\label{tt}
X(u,v)=(g(u)\cosh\theta\sin v + u\sinh\theta, g(u)\cos v+a, g(u)\sinh\theta\sin v+u\cosh\theta).
\end{align*}
\end{enumerate}
\begin{align} \label{ntt}
	\langle \nabla f, N\rangle&= \frac 1{\sqrt{[g'(u)]^2-1}}[(g(u)g'(u)+u)\sinh\theta\cosh\theta\sin v + g(u)\cosh^2\theta\sin^2v\nonumber\\
	&+g(u)\cos^2v+ug'(u)\sinh^2\theta+a\cos v].
\end{align}
\end{lem}

%======================================
\begin{thm}\label{thm3} 
A zero  $f$-mean curvature timelike surface of revolution in $\Bbb G^2\times\Bbb R$ is either a vertical plane containing the $z$-axis, the cylinder $x^2+y^2=1,$ or a timelike $f$-Catenoid.
\end{thm}

\begin{proof}
It is not hard to check that (\ref{ntl}) can not be constant, (\ref{nts}) is constant if and only if $\theta=0$ and $g'=0$ and (\ref{ntt}) is constant if and only if $\theta=a=0.$ That means if $l$ is lightlike, $\Sigma$ must be a vertical plane. By Lemma (\ref{lem1}) such a plane is of zero $f$-mean curvature if and only if  the plane contains the $z$-axis. 
Now consider the case $l$ is timelike.  The condition $\theta=a=0$ means that $l$ is the $z$-axis.
If  $\gamma$ is a vertical line, then  $\Sigma$ is a circular cylinder. By Lemma (\ref{lem1}) a circular cylinder has non-zero $f$-mean curvature and only the cylinder $x^2+y^2=1$ is timelike $f$-minimal. Now consider the case $\gamma$ is not a vertical line. Locally we can parametrize $\Sigma$ as follows.
\begin{align}
X(u,v)=(u\sin v, u\cos v, g(u) ),\ \ 1-g'^2<0.
\end{align}

A direct computation shows that $\Sigma$ is timelike $f$-minimal if and only if $u$ is a solution of the following equation.
\begin{align}\label{tmse}
g''(u)u+g'(u)(1-g'^2(u))+u^2g'(u)(g'^2(u)-1)=0,\end{align}
Equation (\ref{tmse}) is equivalent to
\begin{align*}	g''(u)u=(1-u^2)g'(u)(g'^2(u)-1). \end{align*}
or

\begin{align}\label{tmse1}
	\dfrac{g''(u)}{g'(u)(g'^2(u)-1)}=\dfrac{1}{u}-u.\end{align}
	Intergrating both sides of (\ref{tmse1}), we obtain
\begin{align*}	\ln \dfrac{g'^2(u)-1}{g'^2(u)}=\ln u^2 -u^2 +C,\end{align*}
or 
\begin{align*}	\dfrac{g'^2(u)-1}{g'^2(u)}=\dfrac{Cu^2}{e^{u^2}}
\end{align*}
Hence,
\begin{align*}
	g'(u)=\pm \sqrt{\dfrac{e^{u^2}}{e^{u^2}-Cu^2}}, \ \ e^{u^2}-Cu^2>0.
\end{align*}
Therefore,
\begin{align*}
g(u)=\pm\int_{u_0}^u \sqrt{\dfrac{e^{\tau^2}}{e^{\tau^2}-C\tau^2}}d\tau + B,
\end{align*}
where $e^{u^2}-Cu^2>0$  and $B$ is a constant.
Depending on the value of $C,$ the domain $D$ as well as the initial point $u_0$ are chosen as in the Subsection \ref{4.1}. The surface $\Sigma$ is a timelike $f$-catenoid.
\end{proof}

%======================================================

%===============================================
{\bf Acknowledgements.}  This research is funded by Vietnam National Foundation for Science and Technology Development (NAFOSTED) under grant number 101.04.2014.26.

%===============================

\end{document}